\documentclass[11pt]{amsart}

\usepackage{amsrefs}
\usepackage[all]{xy}
\usepackage{syntonly}
\usepackage{hyperref}
\usepackage{amsfonts}
\usepackage{amssymb}
\usepackage{amsmath}
\usepackage{amsthm}
\usepackage[inline]{enumitem}
\usepackage{tikz}
\usepackage{graphics}
\usepackage{graphicx}
\usepackage{color}
\usepackage{comment}
\usepackage{caption,lipsum}

\usepackage{lineno}

\usepackage{multirow}

\newtheorem{theorem}{Theorem}[section]
\newtheorem{corollary}[theorem]{Corollary}

\newtheorem{definition}[theorem]{Definition}
\newtheorem{lemma}[theorem]{Lemma}

\newtheorem{nonGlobalClaim}{Claim}[theorem]  
\newtheorem{conjecture}[theorem]{Conjecture}

\newtheorem{question}[theorem]{Question}

\newtheorem{fact}[theorem]{Fact}

\specialcomment{com}{ \color{red}\Large }{ \normalsize \color{black}} 
\specialcomment{oldcom}{ \color{brown} }{\color{black}} 

\excludecomment{oldcom} 

\usepackage{stmaryrd}

\begin{document}
\title[]{Adjoining only the things you want:  a survey of Strong Chang's Conjecture and related topics}

\author{Sean Cox}
\email{scox9@vcu.edu}
\address{
Department of Mathematics and Applied Mathematics \\
Virginia Commonwealth University \\
1015 Floyd Avenue \\
Richmond, Virginia 23284, USA 
}

\thanks{The author gratefully acknowledges support from Simons Foundation grant 318467.}


\begin{abstract}
We survey some old and new results on strong variants of Chang's Conjecture and related topics.  
\end{abstract}

\maketitle

\tableofcontents

\section{Introduction}


Variations of the following problem appear frequently in set theory, especially since Shelah's introduction of proper and semiproper forcing.  Given an uncountable set $A$ such that $\omega_1 \subset A$, some Skolemized structure $\mathfrak{A} = (A,\dots)$ in a countable language, some countable $M \subset A$,\footnote{We do not require here that $M \prec \mathfrak{A}$, because some examples of this problem appear when $M \in V$ and $\mathfrak{A}$ is in some forcing extension of $V$.} and some object $a \in A \setminus M$ of interest, we are often interested in adjoining $a$ to $M$, but in a way that doesn't add any new ``unintended" objects.  For example, we often want to know whether we can arrange that the $\mathfrak{A}$-Skolem hull of $M \cup \{ a \}$---which we'll denote $M^{\mathfrak{A}}(a)$---has the property that
\begin{equation}\label{eq_MainProblem}
M^{\mathfrak{A}}(a) \cap \omega_1 = M \cap \omega_1. \tag{+}
\end{equation}
We often informally express \eqref{eq_MainProblem} by saying that ``adjoining $a$ to $M$ doesn't add new countable ordinals".

We will omit the superscript $\mathfrak{A}$ from $M^{\mathfrak{A}}(a)$ when it is clear from the context.  Given a countable $M$, there are always objects $a$ for which the equality \eqref{eq_MainProblem} must fail.  For example, if $a = M \cap \omega_1$, then clearly \eqref{eq_MainProblem} fails.  A slightly less obvious example is when $M \prec \mathfrak{A}$, where $\mathfrak{A}$ is any Skolemized extension of $(H_{\omega_2},\in)$, and $a$ is some ordinal in $\text{sup}(M \cap \omega_2) \setminus M$.  To see that \eqref{eq_MainProblem} must fail in this situation, let $\beta$ be any ordinal in $M \cap \omega_2$ above $a$.  Since $M \prec \mathfrak{A}$ and $\mathfrak{A}$ extends $(H_{\omega_2},\in)$, there is some $f \in M$ that is a surjection from $\omega_1 \to \beta$.  Now $f,a \in M(a)\prec \mathfrak{A}$, so there is some $\xi \in M(a) \cap \omega_1$ such that $f(\xi) = a$.  Then $\xi \notin M$, because otherwise, since $f \in M$, $f(\xi) = a$ would be in $M$ too, contrary to our assumptions about $a$.

So we cannot hope to have \eqref{eq_MainProblem} hold for every choice of $M$ and $a$.  There are several dials to turn to adjust the question, e.g.\ for an arbitrary $M$, for which $a \in A \setminus M$ does \eqref{eq_MainProblem} hold?  Or, for a fixed $a \in A$, and given some (necessarily nonstationary) collection of countable $M \subset A$ such that $a \notin M$, for which such $M$ does equation \eqref{eq_MainProblem} hold?


Such questions come up surprisingly often in set theory.  Here are a few more concrete variants of the question, to give a flavor of how widespread the problem is.

\subsection{Semiproper and Proper forcing (Shelah)}

Suppose that $\mathbb{P}$ is a partial order in $H_\theta$ and $M$ is a countable elementary substructure of $(H_\theta, \in, \mathbb{P})$.  Let $p$ be a condition in $M$.  Can we find a $V$-generic filter $G$, with $p \in G$, such that, letting $\Delta \in V[G]$ be any wellorder of $H^V_\theta[G]$ (so that the resulting structure will have definable Skolem functions), the equation
 \[
 M(G) \cap \omega_1 = M \cap \omega_1 
 \]
 holds?  Here $M(G)$ denotes the hull of $M \cup \{ G \}$ in the structure $(H^V_\theta[G],\in,\Delta)$.  If the answer is ``yes" for every countable $M$ and every $p \in M$, then $\mathbb{P}$ is called \emph{semiproper}.

What if, instead, we make the stronger requirement that 
 \[
 M^{\mathfrak{A}}(G) \cap V = M ?
 \]
If the answer is ``yes" for every countable $M$ and every $p \in M$, then $\mathbb{P}$ is called \emph{proper}.

\subsection{Strong versions of Chang's Conjecture}\label{sec_Intro_SCC}

 For how many $M \in [H_{\omega_3}]^\omega$ is there some $a \in \omega_2 \setminus M$ such that $M(a) \cap \omega_1 = M \cap \omega_1$?  There are always \emph{projective stationarily} many such $M$ (see Section \ref{sec_ZFC_ProjStat}).  Getting club-many such $M$ requires (consistency of) large cardinals, and is a kind of \emph{Strong Chang's Conjecture} discussed in Section \ref{sec_SCC_WRP}.  These strong forms of Chang's Conjecture have interesting characterizations (e.g.\ Theorem \ref{thm_Shelah_SCCcof} and \ref{thm_CoxSakai}), and tend to amplify saturation properties of the nonstationary ideal on $\omega_1$ (Section \ref{sec_ProjCC_NS}).   Moreover, higher variants of this notion (for example, for $M$ of size $\omega_1$) were used by Foreman and Magidor to prove that certain kinds of stationary reflection are inconsistent with ZFC; see Section \ref{sec_Higher}.

Here is a related question that is closely related to stationary set reflection.   Suppose $\theta \ge \omega_2$, $\mathfrak{A}$ is some skolemized extension of $(H_\theta,\in)$, $M$ is countable, and $M$ happens to be of the form $M' \cap W$ for some countable $M' \prec \mathfrak{A}$ and some $W \in M'$ such that $|W|=\omega_1 \subset W \prec \mathfrak{A}$.  Then $M \prec \mathfrak{A}$ (because both $M'$ and $W$ were elementary in $\mathfrak{A}$, and $\mathfrak{A}$ is Skolemized).  Not only do we  have $M(W) \cap \omega_1 = M \cap \omega_1$ in this situation, but in fact
\begin{equation}\label{eq_MWcapW}
M(W) \cap W = M. \tag{**}
\end{equation}
To see the nontrivial direction of \eqref{eq_MWcapW}---i.e.\ the $\subseteq$ direction---let $z \in M(W) \cap W$.  Then $z \in W$ and $z = h(\vec{q})$ for some $\mathfrak{A}$-Skolem function $h$ and some finite tuple $\vec{q}$ from $M \cup \{W \}$.  Since $M \subset M'$ and $W \in M'$, $\vec{q} \in M'$.  Since $M' \prec \mathfrak{A}$, $z=h(\vec{q}) \in M'$.  So $z \in W \cap M' = M$. 

It's natural to ask for how many $M$ does such a $W$ exist:
\begin{question}
For how many $M \in [H_\theta]^\omega$ is it true that there exists a $W \prec \mathfrak{A}$ such that $|W|=\omega_1 \subset W$ and $M(W) \cap W = M$?
\end{question}

It turns out that there are always a large number (``projective stationarily many") of such $M$, and in fact a large number of $M$ for which \emph{stationarily} many $W$'s work.   We will return to this in Section \ref{sec_ZFC_ProjStat}.

\subsection{Antichain catching}\label{sec_IntroAntichainCatching}

Suppose $M$ is a countable elementary substructure of $\mathfrak{A}:=(H_\theta,\in,\Delta)$ and $\mathcal{A}$ is a maximal antichain in the boolean algebra $\wp(\omega_1)/\text{NS}_{\omega_1}$ (where $\text{NS}_{\omega_1}$ denotes the ideal of nonstationary subsets of $\omega_1$). We say that $\boldsymbol{M}$ \textbf{catches} $\boldsymbol{\mathcal{A}}$ if there is some $S \in M \cap \mathcal{A}$ such that $M \cap \omega_1 \in S$.\footnote{More precisely, we should say that there is an $S \in M$ such that $[S] \in \mathcal{A}$ and $M \cap \omega_1 \in S$, where $[S]$ denotes the equivalence class of $S$ in $\wp(\omega_1)/\text{NS}_{\omega_1}$.  We will often omit the equivalence class notation.}

\begin{question}\label{q_HowManyM_saturation}
How many $M$ catch $\mathcal{A}$?
\end{question}
There are always \emph{projective stationarily} many such $M$.  If there are club-many such $M$ (for each $\mathcal{A}$), then $\text{NS}_{\omega_1}$ is \textbf{saturated} (i.e.\ $\wp(\omega_1)/\text{NS}_{\omega_1}$ has the $\omega_2$ chain condition).  The converse holds as well.  This is proved in a highly general form (applicable to ideals other than $\text{NS}_{\omega_1}$) in Lemma 3.46 of Foreman~\cite{MattHandbook}; we sketch a proof just for $\text{NS}_{\omega_1}$ in Section \ref{sec_Prelims}.

Here is a question about antichain-catching that is more closely related to Chang's Conjecture and stationary set reflection is (again, for an arbitrary maximal antichain $\mathcal{A}$ in $\wp(\omega_1)/\text{NS}_{\omega_1}$):
\begin{question}
For how many $M$ does the following hold? 
\begin{equation}\label{eq_ExtendCatch}
\text{There exists some } S \in \mathcal{A} \text{ such that }  M^{\mathfrak{A}}(S) \cap \omega_1 = M \cap \omega_1?
\end{equation}
Equivalently:  how many $M$ can be ``end-extended" (i.e.\ without adding new countable ordinals) to some model that catches $\mathcal{A}$?
\end{question}
 
The stationary reflection principle WRP implies that the answer is ``club-many" (for each $\mathcal{A}$), which in turn implies that $\text{NS}_{\omega_1}$ is \emph{presaturated}.  See Section \ref{sec_Presat}.


\subsection{The scope and purpose of this survey}

This is intended to be a survey of several topics that are closely related to the ``extension" problems described above.  Proofs are generally included if they are sufficiently short, demonstrate some of the common ideas, or simplify/shorten existing proofs in the literature.  There is considerable overlap between this survey and the \emph{Handbook of Set Theory}, especially Foreman's chapter (\cite{MattHandbook}), where these topics are usually treated in much more generality.  The current survey is intended to be more concise, and with a more restricted scope, than those sources.  The survey also includes some newer results (mainly in Section \ref{sec_SCC_WRP}) that have appeared since the \emph{Handbook of Set Theory} was published.  The survey also attempts to uniformize the treatment of some related topics, e.g.\ the ``Global" versions of Strong Chang's Conjecture introduced by Doebler-Schindler~\cite{MR2576698} and Fuchino-Usuba~\cite{FuchinoUsuba} (these are covered in Section \ref{sec_SCC_WRP}).

Section \ref{sec_Prelims} includes preliminaries.  Section \ref{sec_ZFC_ProjStat} covers some basic results in ZFC, generally of the form ``such-and-such a set is always projective stationary".  Section \ref{sec_SCC_WRP}, the longest section of the survey, deals with strong versions of Chang's Conjecture, stationary reflection principles, and related topics.  One can roughly view these as what you get when you replace ``projective stationary" with ``club" in the lemmas from Section \ref{sec_ZFC_ProjStat}.  Section \ref{sec_Higher} covers some results of Foreman-Magidor~\cite{MR1359154} about impossibility of higher stationary set reflection (with attempts to streamline the proof and highlight its connection with Strong Chang's Conjecture). 

\section{Preliminaries}\label{sec_Prelims}

Throughout this paper, we use the word \emph{stationary} in the ``weak" sense of Foreman-Magidor-Shelah~\cite{MR924672} and Larson~\cite{MR2069032}, though in many contexts this ``weak" concept of stationarity is equivalent to Jech's notion of stationarity (see Feng~\cite{MR946635} for a comparison of the two notions).  Namely, a set $S$ is stationary iff for every $F:[\bigcup S]^{<\omega} \to \bigcup S$, there exists an $x \in S$ that is closed under $F$.  This is equivalent to requiring that for every structure $\mathfrak{A}$ in a countable language with universe $\bigcup S$, there exists some $x \in S$ such that $x \prec \mathfrak{A}$.  We will often refer to some ambient space when discussing stationarity, and say things like ``$S$ is stationary \textbf{in} $\boldsymbol{[Z]^{\omega}}$"; by this we mean that $S \subseteq [Z]^{\omega}$, and whenever $\mathfrak{A}$ is a structure on $Z$ in a countable language, then there is a $W \in S$ such that $W \prec \mathfrak{A}$.  Other ambient spaces will sometimes be used as well (e.g.\ $[Z]^{\omega_1}$).  This implies in particular that $\bigcup S = Z$, and hence agrees with the assertion that ``$S$ is stationary" defined above.  We make frequent use of the $\sigma$-completeness of the nonstationary ideal; i.e.\ that a countable union of nonstationary sets is nonstationary.  We also make frequent use of Fodor's Lemma, which asserts that if $f$ is a regressive function on a stationary set $S$---i.e.\ $f(x) \in x$ for every $x \in S$---then there is a stationary $S' \subseteq S$ and a fixed $y$ such that $f(x) = y$ for every $x \in S'$.  The same holds if we replace ``stationary" by ``stationary in such-and-such ambient structure".  Proofs of these and other standard facts about this notion of stationarity appear in Larson~\cite{MR2069032}.  The following lemmas are used frequently:
\begin{lemma}\label{lem_McapOmega1}
Suppose $M \prec (H_\theta,\in)$ and $M$ is countable.  If $D$ contains a club subset of $\omega_1$ and $D \in M$, then $M \cap \omega_1 \in D$.
\end{lemma}
\begin{lemma}
If $\mathfrak{A}=(A,\dots)$ is a structure in a countable signature, and $\mathfrak{A} \in M \prec (H_\theta,\in)$, then $M \cap A \prec \mathfrak{A}$.
\end{lemma}

If $W$ is a set, a \textbf{filtration} of $W$ is a $\subseteq$-continuous and $\subseteq$-increasing sequence $\langle N_i \ :  \ i < |W| \rangle$, with union $W$, such that $|N_i|<|W|$ for all $i < |W|$.  If $M$ and $N$ are sets, we write $\boldsymbol{M \sqsubseteq N}$ iff $M \subseteq N$ and $M \cap \omega_1 = N \cap \omega_1$.  A set $S \subseteq [H_\theta]^\omega$ is \textbf{semistationary} (in $[H_\theta]^\omega$) if 
\[
\{ N \in [H_\theta]^\omega \ : \ \exists M \in S \ M \sqsubseteq N   \}
\]
is stationary.  A partial order is \textbf{proper} if it preserves all stationary subsets of $[H_\theta]^\omega$ (for all large enough $\theta$), and \textbf{semiproper} if it preserves all semistationary subsets of $[H_\theta]^\omega$ (for all large enough $\theta$).

The following lemma is probably the most frequently used lemma in the entire subject. Intuitively, it says that for an uncountable set $W$ and some fixed objects outside of $W$, almost every subset of $W$ can have those new objects adjoined to them, \emph{without} adding new elements of $W$.

\begin{lemma}\label{lem_DontEscapeMcapW}
Suppose $W$ is any uncountable set, $H$ is any superset of $W$, and $\mathfrak{B}$ is a Skolemized structure on $H$ in a countable language.\footnote{In applications, $\mathfrak{B}$ will often include $W$ as a predicate (or even a constant, if $W \in H$).}  Then for ``almost every" $M \in \wp(W)$, 
\[
\text{Hull}^{\mathfrak{B}}(M) \cap W= M.
\]
In other words, letting $C^{\mathfrak{B}}$ denote the set of $M \in \wp(W)$ for which the equation holds, we have that $\wp(W) \setminus C^{\mathfrak{B}}$ is nonstationary in $\wp(W)$.
\end{lemma}

\begin{proof}
Suppose toward a contradiction that $S:=\wp(W) \setminus C^{\mathfrak{B}}$ were stationary in $\wp(W)$ (recall that we are using the notion of ``weak" stationarity).  Then for every $M \in S$, there is some $\mathfrak{B}$-Skolem function $h^M$ and some finite tuple $\vec{q}^M$ from $M$ such that $h^M(\vec{q}^M) \in W \setminus M$.  Since $\mathfrak{B}$ is in a countable language, there are only countably many Skolem functions; so by the $\sigma$-completness of the nonstationary ideal, there is a stationary $S_1 \subseteq S$ and a fixed Skolem function $h$ such that $h^M = h$ for every $M \in S_1$.  Let $n$ denote the arity of $h$.  Then by repeated use of Fodor's Lemma $n$ times (on the regressive maps $M \mapsto \vec{q}^M(0)$, then $M \mapsto \vec{q}^M(1)$, etc.) there is a stationary $S^* \subseteq S_1$ and a fixed $n$-tuple $\vec{q}^{ \ *}$ such that $\vec{q}^{\ *} = \vec{q}^M$ (and hence $h(\vec{q}^{\ *}) \in W \setminus M$) for every $M \in S^*$.  

In summary, $y^*:=h(\vec{q}^{\ *}) \in W \setminus M$ for every $M$ in $S^*$.  And $S^*$ is stationary in $\wp(W)$, which implies $\bigcup S^* = W$.  Since $y^* \in W$, there is some $M \in S^*$ such that $y^* \in M$, a contradiction.
\end{proof}

To illustrate a typical use of Lemma \ref{lem_DontEscapeMcapW}, and because the proof involves simple but powerful techniques that are used so often in this area, we prove the following lemma of Foreman.   Recall the definition of ``catching" an antichain appeared in Section \ref{sec_IntroAntichainCatching}.  The use of Lemma \ref{lem_DontEscapeMcapW} is in the \eqref{item_CatchOne} $\implies$ \eqref{item_Sat} direction of the proof.

\begin{lemma}[Special case of Lemma 3.46 of Foreman~\cite{MattHandbook}]\label{lem_Foreman_CharSat}
The following are equivalent (in what follows, ``maximal antichain" means a maximal antichain in $\wp(\omega_1)/\text{NS}_{\omega_1}$):
\begin{enumerate}
 \item\label{item_Sat} $\text{NS}_{\omega_1}$ is saturated.
 \item\label{item_CatchAll} For every regular $\theta > 2^{\omega_1}$, there are club-many $N \in [H_\theta]^\omega$ such that for every maximal antichain $\mathcal{A} \in N$, $N$ catches $\mathcal{A}$. 
 \item\label{item_CatchOne} For every maximal antichain $\mathcal{A}$ and every regular $\theta > 2^{\omega_1}$, club-many $N \in [H_\theta]^\omega$ catch $\mathcal{A}$.

\end{enumerate}
\end{lemma}
\begin{proof}
\eqref{item_Sat} $\implies$ \eqref{item_CatchAll}:  Assume $\text{NS}_{\omega_1}$ is saturated, and $N \prec (H_\theta,\in)$.  Let $\mathcal{A} \in N$ be a maximal antichain; then $N$ sees that $|\mathcal{A}| \le \omega_1$, and hence that the diagonal union of $\mathcal{A}$ contains a club $D$.  Then $N \cap \omega_1 \in D$, and hence $N \cap \omega_1 \in \bigtriangledown \mathcal{A}$.   It follows that there is some $S \in N \cap \mathcal{A}$ such that $N \cap \omega_1 \in S$.

\eqref{item_CatchAll} $\implies$ \eqref{item_CatchOne}:  Given a particular $\mathcal{A}$, there are club-many $N$ with $\mathcal{A} \in N$.  By assumption, club many of those $N$ catch all of their antichains, so in particular they catch $\mathcal{A}$.

\eqref{item_CatchOne} $\implies$ \eqref{item_Sat}:  assume \eqref{item_CatchOne}.  Let $\mathcal{A}$ be a maximal antichain, and let $\mathfrak{B}=(H_\theta,\dots)$ be a Skolemized structure witnessing that the $\mathcal{A}$-catching sets form a club in $[H_\theta]^\omega$; so for every countable $N \prec \mathfrak{B}$, $N$ catches $\mathcal{A}$.  Let 
\[
W:= \text{Hull}^{\mathfrak{B}}(\omega_1),
\]
and observe that $|W|=\omega_1$.  Suppose for a contradiction that $|\mathcal{A}| \ge \omega_2$; fix some $S \in \mathcal{A} \setminus W$ for the remainder of the proof.  Let $\mathfrak{B}':= \mathfrak{B}^\frown S$.  By Lemma \ref{lem_DontEscapeMcapW}, almost every $N \in [W]^\omega$ has the property that
\[
\text{Hull}^{\mathfrak{B}'}(N) \cap W = N.
\] 
In particular, we can easily find such an $N$ such that, in addition, $N \cap \omega_1 \in S$ and $N \prec \mathfrak{B} | W$ (note that $W$ is elementary in $\mathfrak{B}$, so $\mathfrak{B} | W$ makes sense).  Set $N':= \text{Hull}^{\mathfrak{B}'}(N)$.  Then, in particular, $N \cap \omega_1 = N' \cap \omega_1$; let $\delta$ denote this ordinal.  Furthermore, since $N \prec \mathfrak{B}$, $N$ catches $\mathcal{A}$; so there is some $T \in N \cap \mathcal{A}$ such that $\delta \in T$.  Now $T \in N$ but $S \notin W \supset N$; in particular, $S$ and $T$ are distinct members of the antichain $\mathcal{A}$, and hence $S \cap T$ is nonstationary.  But $S$ and $T$ are both elements of $N'$, and $N' \cap \omega_1 \in S \cap T$.  This contradicts Lemma \ref{lem_McapOmega1}.
\end{proof}

\section{ZFC results:  some common projective stationary sets}\label{sec_ZFC_ProjStat}

 Feng-Jech~\cite{MR1668171} defined a subset $P \subseteq [H_\theta]^\omega$ to be  \emph{projective stationary} iff for every stationary $T \subseteq \omega_1$, the set
\[
P \searrow T:= \{ M \in P \ : \ M \cap \omega_1 \in T \}
\]
is stationary in $[H_\theta]^\omega$.

For the rest of the section, we prove several ZFC results, which often conclude that there are \emph{projective stationarily many} $M \in [H_\theta]^\omega$ with some nice extension property.  As we will see in subsequent sections, to move from \emph{projective stationarily many} to \emph{club many} results in a statement that not only is independent of ZFC, but has large cardinal strength.

The following lemma is the ZFC result alluded to in Section \ref{sec_Intro_SCC} above; it can be viewed as a ZFC-provable version of the principle Global $\text{SCC}^{\text{cof}}_{\text{gap}}$ that will be introduced in Section \ref{sec_SCC_WRP}.  The proof makes use of the notion of an \textbf{internally approachable} set of size $\omega_1$; this is a set $W$ such that there is some $\subseteq$-increasing and continuous sequence $\vec{N}=\langle N_i \ : \ i < \omega_1 \rangle$ of countable sets, with union $W$, such that every proper initial segment of $\vec{N}$ is an element of $W$.  $\text{IA}_{\omega_1}$ denotes the class of sets that are internally approachable of size $\omega_1$.  The following facts are well-known and easy to prove:
\begin{fact}\label{fact_IA}
Suppose $\theta \ge \omega_2$ is regular.
\begin{itemize}
 \item $\text{IA}_{\omega_1} \cap [H_\theta]^{\omega_1}$ is stationary in $[H_\theta]^{\omega_1}$.
 \item If $\omega_2 \le \lambda < \theta$, $\lambda$ is regular, $W \prec (H_\theta,\in)$, and $\lambda \in W \in \text{IA}_{\omega_1}$, then $W \cap H_\lambda \in \text{IA}_{\omega_1}$.
 \item If $W \in \text{IA}_{\omega_1}$ then $W \cap [W]^\omega$ contains a club subset of $[W]^\omega$ (this latter property is called \emph{internally club} by Foreman-Todorcevic~\cite{MR2115072}).
 \item If $W \in \text{IA}_{\omega_1}$ and $W \in M \prec (H_\theta,\in)$, where $M$ is countable, then $M \cap W \in W$ (this really just follows from the internal clubness of $W$).
\end{itemize}
\end{fact}

\begin{lemma}\label{lem_ZFC_GlobalSCC}
Given a regular $\theta \ge \omega_2$ and a Skolemized structure $\mathfrak{A}$ in a countable language extending $(H_\theta,\in)$, there are \textbf{projective-stationarily} many $M \in [H_\theta]^\omega$ such that
\[
\Gamma^{\mathfrak{A}}(M):=\{ W \in [H_\theta]^{\omega_1} \ : \ \omega_1 \subset W \prec \mathfrak{A} \text{ and }   M^{\mathfrak{A}}(W) \cap W = M   \}
\]
is stationary in $[H_\theta]^{\omega_1}$.
\end{lemma}
\begin{proof}
Let $T$ be a stationary subset of $\omega_1$; we need to prove that there are stationarily many $M \in [H_\theta]^\omega$ such that $M \cap \omega_1 \in T$ and $\Gamma^{\mathfrak{A}}(M)$ is stationary in $[H_\theta]^{\omega_1}$.\footnote{In fact, the proof can be modified to show that (for stationarily many $M$) the set $\text{IA}_{\omega_1} \cap \Gamma^{\mathfrak{A}}(M)$ is stationary.}

Suppose toward a contradiction that this fails.  Then there is a Skolemized structure $\mathfrak{B}$ in a countable language, which we can without loss of generality assume extends $\mathfrak{A}$, such that whenever $M \prec \mathfrak{B}$ is countable and $M \cap \omega_1 \in T$, then $\Gamma^{\mathfrak{A}}(M)$ is nonstationary in $[H_\theta]^{\omega_1}$.  For each such $M$, let $\mathfrak{C}_M$ be a Skolemized structure on $H_\theta$ witnessing the nonstationarity of $\Gamma^{\mathfrak{A}}(M)$ in $[H_\theta]^{\omega_1}$.  So whenever $\mathfrak{C}_M$ is defined, and whenever $W$ is a set such that $|W|=\omega_1 \subset W  \prec  \mathfrak{A}$ and $W \prec \mathfrak{C}_M$, then $M^{\mathfrak{A}}(W) \cap W \supsetneq M$. 

Fix a regular $\Omega   >> \theta$, and let
\[
\mathfrak{D}=(H_\Omega,\in,T, \{\mathfrak{A}, \mathfrak{B}, \vec{\mathfrak{C}} \})
\]
where
\[
\vec{\mathfrak{C}}:= \langle \mathfrak{C}_M \ : \ M \in [H_\theta]^\omega, \ M \prec \mathfrak{B},  \text{ and } M \cap \omega_1 \in T  \rangle
\]
Fix a $W' \prec \mathfrak{D}$ such that $|W'|=\omega_1 \subset W'$, and $W' \in \text{IA}_{\omega_1}$; this is possible by Fact \ref{fact_IA}. Set $W:= W' \cap H_\theta$; then by Fact \ref{fact_IA}, $W$ is also in $\text{IA}_{\omega_1}$.  Also notice that 
\[
W \prec \mathfrak{B}
\]
because $\mathfrak{B} \in W'$.

Now fix a countable $M' \prec \mathfrak{D}$ such that $W' \in M'$ and $M'  \cap \omega_1 \in T$.  Set $M:= M' \cap W$.  Then:
\begin{enumerate}
 \item Because $\mathfrak{B} \in M'$ and $W \prec \mathfrak{B}$, and because $\mathfrak{B}$ is Skolemized, it follows that $M = M' \cap W \prec \mathfrak{B}$. Moreover, $M \cap \omega_1 \in T$, because $\omega_1 \subset W$ and $M' \cap \omega_1 \in T$.  Hence $\mathfrak{C}_M$ is defined.
 \item $M = M' \cap W \in W$, by Fact \ref{fact_IA}.
 \item Since $M \in W = W' \cap H_\theta$ and $\vec{\mathfrak{C}} \in W'$, $\mathfrak{C}_M$ is an element of $W'$.  It follows that $W = W' \cap H_\theta \prec \mathfrak{C}_M$.  
\end{enumerate}

We claim that $M^{\mathfrak{A}}(W) \cap W = M$, which will be a contradiction.  For the nontrivial direction ($\subseteq$), notice that an arbitrary element of $M^{\mathfrak{A}}(W) \cap W$ has the form $h(\vec{p},W)$ for some $\mathfrak{A}$-Skolem function $h$ and some parameter $\vec{p} \in M$, and moreover $h(\vec{p},W) \in W$.  Now $\vec{p}$ and $W$ are both elements of $M'$, and $\mathfrak{A} \in M'$; hence $h(\vec{p},W) \in M'$.  So $h(\vec{p},W) \in W \cap M' = M$, completing the proof.

\end{proof}

\begin{corollary}\label{cor_ZFC_adjoin}
For any regular $\theta \ge \omega_2$ and any Skolemized structure $\mathfrak{A}$ on $H_\theta$, there are (at least) projective stationarily many $M \in [H_\theta]^\omega$ such that, for \emph{some} $\alpha \in \omega_2 \setminus M$, $M(\alpha) \cap \omega_1 = M \cap \omega_1$.  
\end{corollary}
\begin{proof}
By Lemma \ref{lem_ZFC_GlobalSCC}, there are projective stationarily many $M \in [H_\theta]^\omega$ such that for stationarily many $W \in [H_\theta]^{\omega_1}$, $\omega_1 \subset W$ and $M(W) \cap W = M$. Fix such an $M$ and $W$ and set $\alpha:=W \cap \omega_2$.  Then $\alpha \notin M$, and because $\alpha$ is definable from $W$ we have $M(\alpha) \subseteq M(W)$.  Since $M(W) \cap W = M$, in particular $M(W) \cap \omega_1 = M \cap \omega_1$, and hence $M(\alpha) \cap \omega_1 = M \cap \omega_1$ too. 
\end{proof}

Can we replace ``projective stationarily many" with ``club-many" in the conclusions of the previous results?  Consistently, yes; but it has large cardinal strength.  This leads us into a hierarchy of \emph{Strong Chang's Conjectures} discussed in Section \ref{sec_SCC_WRP}.

Recall from Section \ref{sec_IntroAntichainCatching} that given a maximal antichain $\mathcal{A}$ in $\wp(\omega_1)/\text{NS}_{\omega_1}$, and a countable $N \prec (H_\theta,\in)$, we say that $\boldsymbol{N}$ \textbf{catches} $\boldsymbol{\mathcal{A}}$ if there is some $S \in \mathcal{A}$ such that $S \in N$ and $N \cap \omega_1 \in S$ (again, by $S \in \mathcal{A}$ we really mean the equivalence class of $S$ is in $\mathcal{A}$).

\begin{lemma}[Feng-Jech~\cite{MR1668171}]\label{lem_ProjStatManyCatch}
Suppose $\mathcal{A}$ is a maximal antichain in $\wp(\omega_1)/\text{NS}_{\omega_1}$, and $\theta$ is a large regular cardinal.  Then there are projective-stationarily many $N \in [H_\theta]^\omega$ that catch $\mathcal{A}$.
\end{lemma}
\begin{proof}
Let $T$ be a stationary subset of $\omega_1$.  Since $\mathcal{A}$ is maximal, there is some $S \in \mathcal{A}$ such that $S \cap T$ is stationary.  Fix any countable $N \prec H_\theta$ with $S,T \in N$ and $N \cap \omega_1 \in S \cap T$.  Then $N$ catches $\mathcal{A}$ (as witnessed by $S$), and $N \cap \omega_1 \in T$.
\end{proof}
Lemma \ref{lem_ProjStatManyCatch}, along with an argument resembling the \ref{item_CatchOne} $\implies$ \ref{item_Sat} direction of the proof of Lemma \ref{lem_Foreman_CharSat}, can be used to show that the \emph{Strong Reflection Principle (SRP)} of \cite{MR1668171} implies that the nonstationary ideal on $\omega_1$ is saturated.  See \cite{MR1668171} for details.

\section{Chang's Conjecture and stationary set reflection}\label{sec_SCC_WRP}

\subsection{Local versions of Strong Chang's Conjecture}\label{sec_LocalSCC}

Given cardinals $\rho < \mu \le \lambda < \kappa$, we write
\[
(\kappa, \lambda) \twoheadrightarrow (\mu, \rho)
\]
to mean that for every structure $\mathfrak{A} = (\kappa,\dots)$ in a countable signature, there is an $X \prec \mathfrak{A}$ such that $|X|=\mu$ and $|X \cap \lambda|=\rho$.  We will mainly be interested in instances of the form
\[
(\mu^{++}, \mu^{+}) \twoheadrightarrow (\mu^+, \mu)
\]
where $\mu$ is an infinite regular cardinal.  For example, the classic \textbf{Chang's Conjecture}, which we'll abbreviate \textbf{CC}, is the principle
\[
(\omega_2,\omega_1) \twoheadrightarrow (\omega_1,\omega).
\]
CC is equiconsistent with an $\omega_1$-Erd\H{o}s cardinal (\cite{MR520190}), and has many combinatorial consequences such as non-existence of Kurepa trees on $\omega_1$, and that every $f: \omega_1 \to \omega_1$ is bounded on a stationary set by some \emph{canonical function}.

It is often convenient to work with more ambient set theory when dealing with Chang's Conjecture, in which case the following lemma (really a special case of Lemma \ref{lem_DontEscapeMcapW}) is useful:
\begin{lemma}[folklore; see e.g \cite{MattHandbook}]\label{lem_CharCC}
Let $\mu$ be a regular cardinal.  The following are equivalent:
\begin{itemize}
 \item $(\mu^{++}, \mu^{+}) \twoheadrightarrow (\mu^+, \mu)$
 \item For every regular $\theta \ge \mu^{++}$, the set
 \[
 \{ X \subset H_\theta \ : \  \text{otp}(X \cap \mu^{++}) = \mu^+ \text{ and } X \cap \mu^+ \in \mu^+   \}
 \] 
 is (weakly) stationary.\footnote{Recall from Section \ref{sec_Prelims}, this means that for every $F: [H_\theta]^{<\omega} \to H_\theta$, there exists an $X$ in the displayed set that is closed under $F$.}
\end{itemize}
\end{lemma}

In order to resolve a question of Baumgartner-Taylor~\cite{MR654852} about ``c.c.c.-indestructible saturation", Foreman-Magidor-Shelah~\cite{MR924672} introduced a stronger form of CC, which we will call \textbf{Projective CC}:
\begin{definition}
\textbf{Projective CC} asserts that ``Chang structures" are projective over $\omega_1$; i.e.\ for every stationary $T \subseteq \omega_1$, the set
\[
\{ X  \subset \omega_2 \ : \  \text{otp}(X \cap \omega_2) = \omega_1 \text{ and } X \cap \omega_1 \in T \}
\]
is (weakly) stationary.  
\end{definition}

Projective CC has a characterization analogous to the characterization of CC in Lemma \ref{lem_CharCC}.  Section \ref{sec_ProjCC_NS} will review some results of Foreman-Magidor-Shelah~\cite{MR924672} and P.\ Larson, showing that Projective CC amplifies the saturation properties (if any exist) of the nonstationary ideal on $\omega_1$.

Other strong variants of CC have appeared in the literature, with inconsistent terminology and notation (see Table 1 in \cite{Cox_Nonreasonable} for a comparison).  We introduce several forms of ``Strong" CC.  In order for this to be applicable to the Foreman-Magidor results in Section \ref{sec_Higher}, we state them in a general form which make sense at higher cardinals.  In what follows, 
\[
\wp^*_\mu(H):= \{ W \subset H \ : \ |W|<\mu \text{ and } W \cap \mu \in \mu \}.
\]
For $\mu = \omega_1$, $\wp^*_\mu(H)$ is essentially the same (mod NS) as what is usually denoted $\wp_{\omega_1}(H)$, but for $\mu \ge \omega_2$ they can consistently differ; the point is that the set $\wp^*_\mu(H)$ does \textbf{not} include ``Chang-type" subsets of $H$.  For example, in the case $\mu = \omega_2$, $\wp^*_{\omega_2}(H)$ does \emph{not} include those $W \subset H$ such that $|W|=\omega_1$ but $|W \cap \omega_1| = \omega$.   One reason for using $\wp^*_\mu(H)$ instead of $\wp_\mu(H)$ on some occasions is that the notions of weak and strong stationarity coincide for subsets of $\wp^*_\mu(H)$ (though not necessarily for subsets of $\wp_\mu(H)$; see Feng~\cite{MR946635}).

We first define some ``local" versions of Strong Chang's Conjecture.

\begin{definition}[local versions of Strong Chang's Conjecture]\label{def_SCC_localversions}
Let $\mu$ be a regular uncountable cardinal.  We define the principles $\text{SCC}(\mu)$, $\text{SCC}^{\text{cof}}(\mu)$, $\text{SCC}^{\text{cof}}_{\text{gap}}(\mu)$, and $\text{SCC}^{\text{split}}(\mu)$ in parallel.  
They assert (respectively) that for all sufficiently large regular $\theta$ and all wellorders $\Delta$ on $H_\theta$ and all $M \prec (H_\theta,\in,\Delta)$ such that $M \in \wp^*_\mu(H_\theta)$:  letting
\[
\text{End}_\mu(M):= \{ M' \prec (H_\theta,\in,\Delta) \ : \  M' \in \wp^*_\mu(H_\theta), \ M \subseteq M', \ \text{ and } M \cap \mu = M' \cap \mu \}
\]
we have:

\begin{itemize}
 \item $\textbf{SCC}\boldsymbol{(\mu)}$:  there exists an $M' \in \text{End}_\mu(M)$ such that $(M' \setminus M) \ \cap \mu^+ \ne \emptyset$.
 \item $\textbf{SCC}^{\textbf{cof}}\boldsymbol{(\mu)}$: there are cofinally many $\gamma < \mu^+$ such that there exists an $M' \in \text{End}_\mu(M)$ such that $\gamma \le \text{sup}(M' \cap \mu^+)$.
 \item $\textbf{SCC}^{\textbf{cof}}_{\textbf{gap}}\boldsymbol{(\mu)}$: there are cofinally many $\gamma < \mu^+$ such that there exists an $M' \in \text{End}_\mu(M)$ such that $\gamma \le \text{sup}(M' \cap \mu^+)$, \emph{and} $M' \cap \gamma = M \cap \gamma$.
 \item $\textbf{SCC}^{\textbf{split}}\boldsymbol{(\mu)}$: there exist $M_0$, $M_1$ in $\text{End}_\mu(M)$ such that $M_0 \cap \mu^+$ and $M_1 \cap \mu^+$ are $\subseteq$-incomparable (i.e., neither is a subset of the other).
\end{itemize}

\textbf{Convention:} If the $\mu$ is not specified, it is intended to be $\omega_1$; e.g., SCC means $\text{SCC}(\omega_1)$.
\end{definition}

For example, in the case $\mu = \omega_1$, SCC (i.e.\ $\text{SCC}(\omega_1)$) asserts that for all large regular $\theta$ and all wellorders $\Delta$ on $H_\theta$ and all countable $M \prec (H_\theta,\in,\Delta)$, there is an $M' \prec (H_\theta,\in,\Delta)$ such that $M \subset M'$, $M \cap \omega_1 = M' \cap \omega_1$, but $M'$ includes some ordinal in $\omega_2 \setminus M$.  By the discussion in the introduction, such an ordinal is necessarily in the interval $\big[\text{sup}(M \cap \omega_2), \omega_2 \big)$.

For $\mu = \omega_1$, all of the variants in Definition \ref{def_SCC_localversions} are consistent relative to a measurable cardinal.\footnote{Cox~\cite{Cox_Nonreasonable} proves that if there is a normal ideal on $\omega_2$ whose quotient forcing is proper---as is the case in $V^{\text{Col}(\omega_1, < \kappa)}$ when $\kappa$ is measurable in $V$ (see \cite{MR560220})---then $\text{SCC}^{\text{cof}}_{\text{gap}}$ holds. }  For $\mu \ge \omega_2$, they all turn out to be inconsistent, though the (inconsistent) principle $\text{SCC}(\omega_2)$ turns out to be a useful intermediary in other inconsistency proofs (this is due to Foreman-Magidor~\cite{MR1359154}; see Section \ref{sec_Higher}).

The following lemma provides a useful characterization of the principle $\text{SCC}(\mu)$, by basically allowing one to turn a single counterexample into stationarily many.  We omit the proof, and refer the reader to the proof of Lemma 13 of \cite{Cox_Nonreasonable}. 
\begin{lemma}\label{lem_StatManyCounterexamples}
For a regular $\mu$, $\text{SCC}(\mu)$ is equivalent to the assertion that for all but nonstationarily many $M \in \wp^*_\mu(H_{\mu^{++}})$, there is an $M' \prec (H_{\mu^{++}}, \in)$ such that $M \subset M'$, $M \cap \mu = M' \cap \mu$, and $(M' \setminus M) \cap \mu^+ \ne \emptyset$.
\end{lemma}

We note that SCC and CC have more similar characterizations than might first be apparent.  Let us call a set $X$ a \emph{Chang set} if $\text{otp}(X \cap \omega_2) = \omega_1$ and $X \cap \omega_1 \in \omega_1$.  Then CC holds iff (for every large $(H_\theta,\in,\Delta)$) there are \emph{stationarily} many $M \in [H_\theta]^\omega$ that can be $\sqsubseteq$-extended to a Chang elementary substructure of $(H_\theta,\in,\Delta)$; while SCC holds iff there are \emph{club} many such $M \in [H_\theta]^\omega$.

The following implications are straightforward (see Cox-Sakai~\cite{Cox_Sakai_SCC}):
\begin{align}\label{eq_SCC_implications}
\begin{gathered}
 \text{SCC}^{\text{cof}}_{\text{gap}}  \implies  \text{SCC}^{\text{cof}}  \implies  \text{SCC}^{\text{split}}  \implies  \text{SCC} \implies \\
  \text{Projective CC}  \implies  \text{CC}.
\end{gathered}
\end{align}

It is known that the implication $\text{SCC}^{\text{cof}}_{\text{gap}} \implies \text{SCC}^{\text{cof}}$ is not reversible (Cox~\cite{Cox_Nonreasonable}).  It is open whether any of the implications between $\text{SCC}^{\text{cof}}$ and SCC are reversible; it is even open whether the implication $\text{SCC}^{\text{cof}} \implies \text{SCC}$ is reversible.  Those questions are related to Conjecture \ref{conj_SCCcof_SCCsplit} below.    

Regarding the remaining implications from \eqref{eq_SCC_implications}, Todorcevic~\cite{MR1261218} observed that SCC implies that every stationary subset of $[\omega_2]^\omega$ reflects to an ordinal in the interval $(\omega_1,\omega_2)$.  Such a reflection property fails after adding a Cohen real $\sigma$, because Gitik~\cite{MR820120} proved that $S:=V[\sigma] \setminus V$ is stationary in $[\omega_2]^\omega$ in $V[\sigma]$.  And $S$ cannot reflect to any ordinal $\gamma \in (\omega_1, \omega_2)$, because $V \cap [\gamma]^\omega$ contains a club (just fix any $\omega_1$-length filtration of $\gamma$ in $V$).  In short, SCC fails after adding a Cohen real.  The following lemma (a slight extension of the well-known theorem that CC is preserved by c.c.c.\ forcing) shows that, on the other hand, Projective CC is preserved by such forcing:
\begin{lemma}
Projective CC is preserved by c.c.c.\ forcing.
\end{lemma}
\begin{proof}
Suppose $\mathbb{P}$ is c.c.c., $\dot{F}$ is a $\mathbb{P}$-name for a function from $[\omega_2]^{<\omega} \to \omega_2$, and $\dot{T}$ is a $\mathbb{P}$-name for a stationary subset of $\omega_1$.  Let $p$ be a condition.  Since $\mathbb{P}$ preserves $\omega_1$, there are stationarily many $\alpha < \omega_1$ such that some condition $p(\alpha)$ below $p$ forces $\check{\alpha} \in \dot{T}$.  Let $S$ denote this stationary set; by Projective CC there is an $X \prec (H_\theta,\in, \mathbb{P},p,\dot{T},\dot{F})$ such that $\alpha_X:=X \cap \omega_1 \in S$.  Let $G$ be generic with $p(\alpha_X) \in G$.  Then $\alpha_X \in T:= \dot{T}_G$ and $X[G]$ is closed under $F:=\dot{F}_G$.  Since $\mathbb{P}$ was c.c.c., $1_{\mathbb{P}}$ is a master condition for every elementary submodel (countable or otherwise), in particular for $X$. So $X[G] \cap V = X$.  So $|X[G] \cap \omega_2|=\omega_1$ and $X[G] \cap \omega_1 = X \cap \omega_1 = \alpha_X \in T$.
\end{proof}

So the implication from SCC to Projective CC is not reversible, because the latter is preserved by adding a Cohen real but the former is not.  Finally, Projective CC is known to have strictly higher consistency strength than CC (see Sharpe-Welch~\cite{MR2817562}).

The reversibility of the remaining implications in \eqref{eq_SCC_implications} are all open, but the following theorems may be relevant.  Shelah proved an interesting characterization of $\text{SCC}^{\text{cof}}$:
\begin{theorem}[Shelah]\label{thm_Shelah_SCCcof}
The following are equivalent:
\begin{enumerate}
 \item $\text{SCC}^{\text{cof}}$.
 \item Namba forcing is semiproper.
 \item There exists some semiproper poset that forces $\text{cf}(\omega_2^V) = \omega$.
\end{enumerate}
\end{theorem}

Most of the implications of Theorem \ref{thm_Shelah_SCCcof} are proven in Chapter XII of Shelah~\cite{MR1623206}; for the proof that $\text{SCC}^{\text{cof}}$ implies semiproperness of Namba forcing, see Section 3 of Doebler~\cite{MR3065118}.

Cox and Sakai proved a characterization of $\text{SCC}^{\text{split}}$ that closely mimics Shelah's Theorem \ref{thm_Shelah_SCCcof}:
\begin{theorem}[Cox-Sakai~\cite{Cox_Sakai_SCC}]\label{thm_CoxSakai}
The following are equivalent:
\begin{enumerate}
 \item $\text{SCC}^{\text{split}}$
 \item The poset that adds a Cohen real, then shoots a club through $([\omega_2]^\omega) \setminus V$ with countable conditions, is semiproper.
 \item There exists some semiproper poset that forces $([\omega_2]^\omega)^V$ to be nonstationary.
\end{enumerate}
\end{theorem}

In light of Shelah's Theorem \ref{thm_Shelah_SCCcof} and the Cox-Sakai Theorem \ref{thm_CoxSakai}, we make the following conjecture:
\begin{conjecture}\label{conj_SCCcof_SCCsplit}
The implication $\text{SCC}^{\text{cof}} \implies \text{SCC}^{\text{split}}$ is not reversible.
\end{conjecture}

\subsection{Global versions of Strong Chang's Conjecture}

We now introduce ``global" versions of $\text{SCC}^{\text{cof}}$ and $\text{SCC}^{\text{cof}}_{\text{gap}}$, because they are (respectively) equivalent to reflection principles.  The principles \textbf{Global} $\textbf{SCC}^{\textbf{cof}}$ and \textbf{Global} $\textbf{SCC}^{\text{cof}}_{\text{gap}}$ were introduced by Doebler-Schindler~\cite{MR2576698} and Fuchino-Usuba~\cite{FuchinoUsuba}, respectively (but under different names).  Unlike Definition \ref{def_SCC_localversions} we will only need the version for $\mu = \omega_1$.  Note also the similarity of the following definition with Lemma \ref{lem_ZFC_GlobalSCC}.
\begin{definition}[``Global" versions of Strong Chang's Conjecture]\label{def_SCC_Global}
We define ``global" versions of $\text{SCC}^{\text{cof}}$ and $\text{SCC}^{\text{cof}}_{\text{gap}}$.  They assert (respectively) that for all sufficiently large regular $\theta$ and all wellorders $\Delta$ on $H_\theta$ and all countable $M \prec  \mathfrak{A}:=(H_\theta,\in,\Delta)$: \begin{itemize}
 \item $\textbf{Global } \textbf{SCC}^{\textbf{cof}}_{\text{gap}}$:  the set
\[
\Gamma^{\mathfrak{A}}(M):= \{ W \in \wp^*_{\omega_2}(H_\theta) \ :  \   M^{\mathfrak{A}}(W) \cap W = M  \}
\]
is $\subseteq$-cofinal in $\wp^*_{\omega_2}(H_\theta)$.

 \item $\textbf{Global } \textbf{SCC}^{\textbf{cof}}$:  the set
\[
\Gamma^{\mathfrak{A}}_{\sqsubseteq}(M):= \{ W \in \wp^*_{\omega_2}(H_\theta) \ :  \   M^{\mathfrak{A}}(W) \cap W \sqsupseteq M  \}
\]
is $\subseteq$-cofinal in $\wp^*_{\omega_2}(H_\theta)$.
\end{itemize}
\end{definition}

The Global versions easily imply the versions from Definition \ref{def_SCC_localversions}.  For example, if Global $\text{SCC}^{\text{cof}}_{\text{gap}}$ holds, and $M \prec (H_\theta,\in,\Delta)$ is countable, then given any $\gamma < \omega_2$ we can use the Global $\text{SCC}^{\text{cof}}_{\text{gap}}$ assumption to find a $W \in \wp^*_{\omega_2}(H_\theta)$ such that $\gamma < W \cap \omega_2$ and $M(W) \cap W = M$.  It follows that $\gamma < W \cap \omega_2 \in M(W)$, and 
\[
M(W) \cap \gamma = M(W) \cap W \cap \gamma = M \cap \gamma.
\]
Hence $M(W)$ is the end-extension of $M$ required by $\text{SCC}^{\text{cof}}_{\text{gap}}$.

Each principle in Definition \ref{def_SCC_Global} is equivalent to a kind of global \emph{stationary reflection} principle, as described in the next section.

\subsection{Relationship with Stationary reflection principles}\label{sec_RelationSCC_WRP}

The following kind of stationary set reflection (in the case $\mu = \omega_1$) was introduced by Beaudoin~\cite{MR877870} and Foreman-Magidor-Shelah~\cite{MR924672}:
\begin{definition}
For a regular uncountable cardinal $\mu$, the principle $\text{WRP}\big( \wp^*_{\mu} \big)$ asserts that for every regular $\theta \ge \mu^+$ and every stationary $S \subseteq \wp^*_\mu(\theta)$, there is an $W \in \wp^*_{\mu^+}(\theta)$ such that $S \cap \wp^*_{\mu}(W)$ is stationary.

\textbf{Convention:} The unadorned version is understood to mean the version where $\mu = \omega_1$; i.e.\ WRP means $\text{WRP}(\wp^*_{\omega_1})$.
\end{definition}

So, for example, WRP (i.e.\ $\text{WRP}(\wp^*_{\omega_1})$) means that for every regular $\theta \ge \omega_2$ and every stationary $S \subseteq [\theta]^\omega$, there is a $W \subset \theta$ such that $|W|=\omega_1 \subset W$ and $S \cap [W]^\omega$ is stationary in $[W]^\omega$.

\begin{theorem}\label{thm_WRP_implies_SCC}
Let $\mu$ be a regular uncountable cardinal.  The principle $\text{WRP}\big( \wp^*_{\mu} \big)$ implies $\text{SCC}(\mu)$.
\end{theorem}
\begin{proof}
Suppose toward a contradiction that $\text{SCC}(\mu)$ fails; then by Lemma \ref{lem_StatManyCounterexamples}, there is a stationary $S \subseteq \wp^*_{\mu}(H_{\mu^{++}})$ such that for all $M \in S$, there is no $M' \in \text{End}_\mu(M)$ (using the notation from Definition \ref{def_SCC_localversions}) such that $M'$ properly extends $M$ below $\mu^+$.

By $\text{WRP}(\wp^*_\mu)$ there is a $W \in \wp^*_{\mu^+}(H_{\mu^{++}})$ such that $S_W:=S \cap \wp^*_{\mu}(W)$ is stationary in $\wp^*_{\mu}(W)$.  Fix such a $W$ for the remainder of the proof.  Since $S_W$ is stationary in $\wp^*_{\mu}(W)$, by Lemma \ref{lem_DontEscapeMcapW} there is an $M \in S_W$ such that
\[
M(W) \cap W = M
\]
where $M(W)$ denotes the hull of $M \cup \{ W \}$ in the structure $(H_{\mu^{++}}, \in, \Delta)$ (where $\Delta$ is any wellorder of $H_{\mu^{++}}$).  In particular, since $\mu \subset W$, it follows that $M(W) \cap \mu = M \cap \mu$,\footnote{This is where we needed to know that $W$ had transitive intersection with $\mu$; i.e.\ why we require that the reflecting set $W$ is in $\wp^*_{\mu^+}(-)$ rather than just in $\wp_{\mu^+}(-)$.}  So $M(W) \in \text{End}_\mu(M)$.  But also $W \cap \mu^+ \in M(W)$, and $W \cap \mu^+$ is at least as large as $\text{sup}(M \cap \mu^+)$, because $M \subset W$.  Hence $M(W)$ properly end extends $M$ below $\mu^+$.  This contradicts that $M \in S$.

Then, letting $M':= M(W)$, we have a contradiction to the fact that $M \in S$.
\end{proof}

Theorem \ref{thm_WRP_implies_SCC} actually follows from a weaker assumption (see Theorem \ref{thm_DoeblerSchindler} below), but we chose to sketch the proof of Theorem \ref{thm_WRP_implies_SCC} under non-optimal hypotheses, for a couple of reasons.  Firstly, it is all that we need for its main application in Section \ref{sec_Higher}.  Secondly, it highlights what the author considers to be an interesting open problem.  Notice that (in the case $\mu = \omega_1$, for simplicity) the proof actually shows that WRP implies that for every large regular $\theta$ and almost every $M \in [H_\theta]^\omega$, there is a $W \in \wp^*_{\mu^+}(H_\theta)$ such that $M(W) \cap W = M$.  This seems awfully close to getting Global $\text{SCC}^{\text{cof}}_{\text{gap}}$, but in order to obtain the latter, one seems to need that the $M$ from the proof is also an \emph{element} of $W$, so that any purported bound on $\Gamma^{\mathfrak{A}}(M)$ (using the notation from Definition \ref{def_SCC_Global}) would be an element of $W$, and hence $W$ would be beyond this bound, leading to a contradiction.  But it is not clear that we can arrange that $M \in W$ from WRP alone.  This was the apparent motivation of the principle $\textbf{RP}_{\textbf{internal}}$ introduced by Fuchino-Usuba~\cite{FuchinoUsuba} (though under a different name); this principle asserts that for all regular $\theta \ge \omega_2$ and all stationary $S \subseteq \wp_{\omega_1}(H_\theta)$, there is a $W \in \wp^*_{\omega_2}(H_\theta)$ such that  $S \cap W \cap \wp_{\omega_1}(W)$---not merely $S \cap \wp_{\omega_1}(W)$---is stationary in $\wp_{\omega_1}(W)$.  Fuchino and Usuba proved:
\begin{theorem}[Fuchino-Usuba~\cite{FuchinoUsuba}]
\[
\text{RP}_{\text{internal}} \ \iff \ \text{Global } \text{SCC}^{\text{cof}}_{\text{gap}}.
\]
\end{theorem}
Now clearly $\text{RP}_{\text{internal}} \implies \text{WRP}$, but whether this implication is actually an equivalence is open.  More details on these and related problems can be found in Cox~\cite{Cox_RP_IS}.

We mentioned above that the assumptions of Theorem \ref{thm_WRP_implies_SCC} were not optimal.  The optimal result is due to Doebler and Schindler, and involves the \emph{Semistationary Set Reflection Principle (SSR)}, which is weaker than WRP, but still quite strong:
\begin{theorem}[Doebler-Schindler~\cite{MR2576698}]\label{thm_DoeblerSchindler}
\[
\text{SSR} \ \iff \ \text{Global } \text{SCC}^{\text{cof}}.
\]
\end{theorem}
They also obtained several other interesting statements that are also equivalent to Global $\text{SCC}^{\text{cof}}$, e.g.\ the assertion (famously introduced in \cite{MR924672}) that every $\omega_1$-stationary set preserving forcing is semiproper.

\subsection{Strong Chang's Conjecture and the Tree Property}\label{sec_SCC_TP}
The principle $\text{SCC}^{\text{cof}}$ and its global version found applications in recent work of Torres-P{\'e}rez and Wu.  The principle $\text{TP}(\omega_2)$ asserts that there are no $\omega_2$-Aronszajn trees.  The principle $\text{ITP}(\omega_2)$ is a strengthening of $\text{TP}(\omega_2)$  introduced by Weiss~\cite{Weiss_CombEssence}, whose definition we will not give here. 
\begin{theorem}[Torres-P{\'e}rez and Wu]\label{thm_TorresPerezWu}
Assume that the Continuum Hypothesis (CH) fails.  Then:
\begin{enumerate}[label=(\alph*)]
 \item\label{item_NegCH_SCCcof} $\text{SCC}^{\text{cof}}$ implies $\text{TP}(\omega_2)$ (\cite{MR3431031}). 
 \item Global $\text{SCC}^{\text{cof}}$ implies $\text{ITP}(\omega_2)$ (\cite{MR3600760}).
\end{enumerate}
\end{theorem}

The assumption that CH fails in Theorem \ref{thm_TorresPerezWu} cannot be removed, since Global $\text{SCC}^{\text{cof}}$ is consistent with CH, and CH implies failure of $\text{TP}(\omega_2)$ and $\text{ITP}(\omega_2)$.

Todorcevic~\cite{MR1261218} proved that SCC implies every stationary subset of $[\omega_2]^\omega$ reflects to a set of size $\omega_1$ (this is a local fragment of $\text{WRP}(\wp^*_{\omega_1})$).  So in light of Theorem \ref{thm_TorresPerezWu}, the following question from \cite{MR3431031} is a natural one:
\begin{question}
Suppose CH fails and every stationary subset of $[\omega_2]^\omega$ reflects to a set of size $\omega_1$.  Must $\text{TP}(\omega_2)$ hold?
\end{question}

\subsection{WRP and presaturation}\label{sec_Presat}

We now return, yet again, to the notion of antichain catching introduced in Section \ref{sec_IntroAntichainCatching}.  We say that $\text{NS}_{\omega_1}$ is \textbf{presaturated} iff whenever $\langle A_n \ : \ n < \omega \rangle$ is an $\omega$-sequence of maximal antichains in $\wp(\omega_1)/\text{NS}_{\omega_1}$, there are densely many $T$ (i.e.\ densely many stationary sets in the boolean algebra $\wp(\omega_1)/\text{NS}_{\omega_1}$) such that for every $n < \omega$, $T$ is compatible with at most $\omega_1$ many members of $A_n$.  Presaturation suffices for many of the applications of saturation; in particular, presaturation yields ``generic almost huge embeddings" (see \cite{MattHandbook}).

  The following theorem is not optimal; the weaker \emph{Semistationary Reflection Principle} suffices instead of WRP.  But the idea is similar. 

\begin{theorem}[\cite{MR924672}]
WRP implies that $\text{NS}_{\omega_1}$ is presaturated.
\end{theorem}
\begin{proof}
Assume WRP.  We need an end-extension claim.
\begin{nonGlobalClaim}\label{clm_EndExendCatch}
For every maximal antichain $A$, every sufficiently large regular $\theta$, and every wellorder $\Delta$ on $H_\theta$:  whenever $N \prec (H_\theta,\in,\Delta)$, $N$ can be $\sqsubset$-extended to a countable elementary substructure of $(H_\theta,\in,\Delta)$ that catches $A$.

Equivalently: there is an $S \in A$ such that $N \cap \omega_1 \in S$ and
\[
\text{Hull}^{(H_\theta,\in,\Delta)}(N \cup \{S \}) \cap \omega_1 = N \cap \omega_1. 
\]
\end{nonGlobalClaim}
\begin{proof}
(of Claim \ref{clm_EndExendCatch}).  Let $A$ be a maximal antichain, and suppose the claim fails.  Then Lemma \ref{lem_DontEscapeMcapW} can be used to show there are stationarily many $N \in [H_\theta]^\omega$ (for some large $\theta$) for which it fails.  Let $R$ denote this stationary set.  By WRP, there is a $W \in \wp^*_{\omega_2}(H_\theta)$ such that $R \cap [W]^\omega$ is stationary in $[W]^\omega$.  Fix a filtration
\[
\vec{N} = \langle N_i \ : \ i < \omega_1 \rangle
\]
of $W$.  Then 
\[
T_R:= \{ i < \omega_1 \ : \ N_i \cap \omega_1 = i  \text{ and } N_i \in R \} \text{ is stationary in } \omega_1.
\]
Since $A$ is a maximal antichain, there is some $S \in A$ such that $S \cap T_R$ is stationary.  Then
\[
P:=\{ N_i \ : \ i \in T_R \cap S \} \text{ is stationary in } [W]^\omega.
\]
Then by Lemma \ref{lem_DontEscapeMcapW}, there is an $N_i \in P$ such that
\[
\text{Hull}^{(H_\theta,\in,\Delta,S)}(N_i) \cap W = N_i.
\]
Hence, letting $N'_i:= \text{Hull}^{(H_\theta,\in,\Delta,S)}(N_i)$, we have $S \in N'_i$ and $N'_i \cap \omega_1 = N_i \cap \omega_1 = i \in S \cap T_R$.  So $S$ witnesses that $N'_i$ catches $A$.
\end{proof}

Now assume $\langle A_n \ : \ n < \omega \rangle$ is an $\omega$-sequence of maximal antichains.  Let $T$ be a stationary subset of $\omega_1$; we need to find a stationary subset of $T$ such that for each $n$, the subset is compatible with at most $\omega_1$ many members of $A_n$.

Repeated application of Claim \ref{clm_EndExendCatch} $\omega$-many times easily yields:
\begin{nonGlobalClaim}\label{clm_CatchEveryAn}
Fix a large regular $\lambda$.  Then
\[
\{ M \in [H_\lambda]^\omega \ : \ M \cap \omega_1 \in T \text{ and } M \text{ catches every } A_n \}
\]
is stationary.
\end{nonGlobalClaim}
%
 
Let $R_T$ denote the stationary subset of $[H_\lambda]^\omega$ given by Claim \ref{clm_CatchEveryAn}. By WRP, $R_T$ reflects to some $W \prec (H_\lambda,\in,\Delta, T,\vec{A})$ such that $|W|=\omega_1 \subset W$.  Let $\vec{N} = \langle N_i \ : \ i < \omega_1 \rangle$ be a filtration of $W$.  Then
\[
T':=\{ i < \omega_1 \ : \  N_i \cap \omega_1 = i \text{ and } N_i \in R_T  \}
\]
is a stationary subset of $T$.  The following claim will finish the proof (this is yet another proof that resembles the \ref{item_CatchOne} $\implies$ \ref{item_Sat} direction of Lemma \ref{lem_Foreman_CharSat}):
\begin{nonGlobalClaim}\label{clm_ContainedInW}
For every $n < \omega$, 
\[
\{ S \in A_n \ : \ T' \cap S \text{ is stationary} \} \subset W.
\]
\end{nonGlobalClaim}
\begin{proof}
(of Claim \ref{clm_ContainedInW}):  Suppose for a contradiction that for some $n < \omega$ and some $S \in A_n \setminus W$, $T' \cap S$ is stationary.  Then
\[
G:=\{ N_i \ : \  N_i \cap \omega_1 = i \in T' \cap S  \}
\]
is a stationary subset of $[W]^\omega$.  By Lemma \ref{lem_DontEscapeMcapW}, there is an $N_i \in G$ such that
\[
\text{Hull}^{(H_\lambda,\in,\Delta, S)}(N_i) \cap W = N_i.
\] 
Now since $N_i \in R$, $N_i$ catches $A_n$; so fix some $S_1 \in A_n$ witnessing this.  Note that $S \ne S_1$ because $S_1 \in N_i \subset W$ but $S \notin W$.  Let $N'_i:= \text{Hull}^{(H_\lambda,\in,\Delta, S)}(N_i)$.  Then, in particular, $N'_i \cap \omega_1 = N_i \cap \omega_1 \in S \cap S_1$. But $S$ and $S_1$ are both elements of  $N'$, and are distinct members of the antichain $A_n$, so $S \cap S_1$ is a nonstationary element of $N'$.  Since $N' \cap \omega_1 \in S \cap S_1$, this contradicts Lemma \ref{lem_McapOmega1}.
\end{proof}
\end{proof}

\subsection{Forcing properties of sealing forcings}

Given a maximal antichain $\mathcal{A}$, the \textbf{sealing forcing for $\mathcal{A}$} (defined by Foreman-Magidor-Shelah~\cite{MR924672}) is the poset $\text{Col}(\omega_1,\mathcal{A})$ followed by shooting a club (using initial segments) through the diagonal union of $\mathcal{A}$.  An equivalent way to represent this forcing is as the set of all pairs $(f,c)$ such that:
\begin{itemize}
 \item $f: \gamma \to \mathcal{A}$ for some $\gamma < \omega_1$;
 \item $c$ is a closed, bounded subset of $\omega_1$ such that 
 \[
 \forall \alpha \in c \ \exists i < \alpha \ \ \alpha \in f(i).
 \]
\end{itemize}
A condition $(f',c')$ is stronger than $(f,c)$ iff $f' \supset f$ and $c'$ end-extends $c$.

We will let $\mathbb{S}_{\mathcal{A}}$ denote this poset.  Foreman-Magidor-Shelah~\cite{MR924672} proved that $\mathbb{S}_{\mathcal{A}}$ always preserves stationary subsets of $\omega_1$; this was used in the proof that MM implies saturation of $\text{NS}_{\omega_1}$.

If $\mathbb{S}_{\mathcal{A}}$ is semiproper for every maximal antichain $\mathcal{A}$, then $\text{NS}_{\omega_1}$ is presaturated; the argument is similar to the proof that WRP (or even SSR) implies presaturation.  


When can $\mathbb{S}_{\mathcal{A}}$ be proper?  Certainly if $|\mathcal{A}|\le \omega_1$ it is easy to see that $\mathbb{S}_{\mathcal{A}}$ is proper (in fact, equivalent to a $\sigma$-closed forcng).  M.\ Eskew asked the author if $\mathbb{S}_{\mathcal{A}}$ could ever be proper when $|\mathcal{A}| > \omega_1$.  It cannot; in fact:
\begin{lemma}
Let $\mathcal{A}$ be a maximal antichain in $\wp(\omega_1)/\text{NS}_{\omega_1}$.  The following are equivalent:
\begin{enumerate}
 \item\label{item_SmallAntichain} $|\mathcal{A}| \le \omega_1$.
 \item\label{item_SealingIsSigmaClosed} $\mathbb{S}_{\mathcal{A}}$ is forcing equivalent to a $\sigma$-closed poset.
 \item\label{item_SealingIsProper} $\mathbb{S}_{\mathcal{A}}$ is a proper forcing.

\end{enumerate}
\end{lemma}
\begin{proof}
The implication \ref{item_SmallAntichain} $\implies$ \ref{item_SealingIsSigmaClosed} is straightforward, and left to the reader.  The implication \ref{item_SealingIsSigmaClosed} $\implies$ \ref{item_SealingIsProper} is trivial.

For the \ref{item_SealingIsProper} $\implies$ \ref{item_SmallAntichain} direction:  suppose $\mathbb{S}_{\mathcal{A}}$ is proper.  The sealing forcing is always $\sigma$-distributive; so in fact $\mathbb{S}_{\mathcal{A}}$ is \emph{totally} proper. In other words, for all large regular $\theta$ and all countable $M \prec (H_\theta,\in,\mathcal{A})$, every condition in $M$ can be extended to a condition whose upward closure generates an $(M,\mathbb{S}_{\mathcal{A}})$-generic filter (i.e.\ a filter that meets $D \cap M$ whenever $D \in M$ and $D$ is dense).  We will call such a condition a \emph{totally generic condition} for $M$.  See Abraham~\cite{MR2768684} for these basic facts about these notions.

Fix any such $M$, and let $(f,c)$ be a totally generic condition for $M$.  An easy density argument yields that $M \cap \omega_1 \subseteq \text{dom}(f)$, and $M \cap \omega_1$ is a limit point, and hence element, of the closed set $c$.  Then by the definition of what it means to be a condition, there is some $i < M \cap \omega_1$ such that $M \cap \omega_1 \in f(i)$.  Now $f \restriction (i + 1) \in M$, and hence $f(i) \in M$; so $M$ catches $\mathcal{A}$.

Since $M$ was arbitrary, this shows that club-many $M \in [H_\theta]^\omega$ catch $\mathcal{A}$.  By the same argument as the  \eqref{item_CatchOne} $\implies$ \eqref{item_Sat} direction of the proof of Lemma \ref{lem_Foreman_CharSat}, $\mathcal{A}$ must have cardinality $\le \omega_1$.
\end{proof}

\subsection{Projective CC and saturation of the nonstationary ideal}\label{sec_ProjCC_NS}

In this section we return to the notion ``Projective CC" introduced earlier, and present two results---the older Theorem \ref{thm_IndestSat} and the newer Theorem \ref{thm_Larson_VerySat}---that demonstrate how Projective CC amplifies saturation properties of the nonstationary ideal on $\omega_1$.

\begin{theorem}[Foreman-Magidor-Shelah~\cite{MR924672}]\label{thm_IndestSat}
Suppose $\text{NS}_{\omega_1}$ is saturated, and Projective CC holds.  Then the saturation of $\text{NS}_{\omega_1}$ is ``c.c.c.-indestructible"; i.e.\ every c.c.c.\ forcing extension satisfies that $\text{NS}_{\omega_1}$ is saturated.
\end{theorem}

To prove Theorem \ref{thm_IndestSat}, we will need the following special case of Foreman's Duality Theorem (this special case was originally proved independently by Kakuda and Magidor; see Corollary 7.17 of \cite{MattHandbook}):
\begin{theorem}\label{thm_SuffPreserveSat}
Suppose $\text{NS}_{\omega_1}$ is saturated and $\mathbb{P}$ is c.c.c.  Let $\dot{\pi}$ be the $\wp(\omega_1)/\text{NS}_{\omega_1}$-name for the generic ultrapower embedding.  If $\wp(\omega_1)/\text{NS}_{\omega_1}$ forces that $\dot{\pi}(\mathbb{P})$ is $\omega_2^V$-cc in the generic extension of $V$ by $\wp(\omega_1)/\text{NS}_{\omega_1}$, then
\[
V^{\mathbb{P}} \models \  \text{NS}_{\omega_1} \ \text{ is saturated.}
\]
\end{theorem}

We now return to the proof of the Foreman-Magidor-Shelah Theorem \ref{thm_IndestSat}:
\begin{proof}
Let $\mathbb{P}$ be c.c.c.  By Theorem \ref{thm_SuffPreserveSat}, it suffices to show that $\wp(\omega_1)/\text{NS}_{\omega_1}$ forces that $\dot{\pi}(\mathbb{P})$ is $\omega_2^V$-cc.  Suppose toward a contradiction that $T$ is a stationary subset of $\omega_1$, $\dot{A}$ is a $\wp(\omega_1)/\text{NS}_{\omega_1}$-name, and
\[
T \Vdash_{\wp(\omega_1)/\text{NS}_{\omega_1}} \ \dot{A} \text{ is an } \omega_2^V\text{-sized antichain in } \dot{\pi}(\mathbb{P}).
\]

By Projective CC, there is an 
\[
X \prec (H_\theta,\in,T,\dot{A})
\]
such that $X \cap \omega_1 \in T$ and $\text{otp}(X \cap \omega_2) = \omega_1$.  Let $\sigma: H_X \to X \prec H_\theta$ be the inverse of the transitive collapsing map of $X$, and let $(\dot{A}_X,\mathbb{P}_X, T_X):= \sigma^{-1}(\dot{A},\mathbb{P},T)$.  Let $\delta:= \text{crit}(\sigma)$; note $\delta = \omega_1^{H_X}$. Since $\text{NS}_{\omega_1}$ is saturated, $X$ catches all of its antichains; this is similar to the argument of the \eqref{item_Sat} implies \eqref{item_CatchAll} direction of Lemma \ref{lem_Foreman_CharSat}.  It follows that
\[
U:= \{ A \in \wp^{H_X}\big( \delta \big) \ : \  \delta \in \sigma(A)  \}
\]
is \emph{generic} over $H_X$ for $\sigma^{-1}\big( \wp(\omega_1)/\text{NS}_{\omega_1} \big)$.  

Let $\pi_U: H_X \to_U N_U$ be the ultrapower of $H_X$ by $U$; by standard arguments, the map $k$ defined by 
\[
k \big(\pi_U(f)(\delta)\big):= \sigma(f)(\delta)
\]
(for any $f \in H_X \cap {}^{\delta} H_X$) is a well-defined, elementary map from $N_U \to H_\theta$, and has the property that $\sigma = k \circ \pi_U$.

Now since $U$ is generic over $H_X$, $H_X[U]$ sees the map $\pi_U$, and believes that it is a generic ultrapower.  Furthermore, since $X \cap \omega_1 \in T$, $T_X \in U$, and so $H_X[U]$ believes that $A:=(\dot{A}_X)_U$ is an antichain in 
\[
\pi_U (\mathbb{P}_X) = k^{-1}(\mathbb{P})
\]
of size $\aleph_2$.  Note that since $X \cap \omega_2$ has ordertype $\omega_1$, $\omega_1^V = \omega_2^{H_X}$.  So, from the point of view of $H_X[U]$, $A$ is an antichain in $k^{-1}(\mathbb{P})$ that has an enumeration of length $\omega_1^V = \omega_2^{H_X}$.  Now although $A$ is not an element of $N_U$, it is a subset of $N_U$, and distinct conditions from $A$ are incompatible in $k^{-1}(\mathbb{P})$.  Then by elementarity of $k: N_U \to H_\theta$, $k " A$ is a collection of pairwise incompatible elements of $\mathbb{P}$.  But $k " A$ has size $\omega_1$ in $V$, contradicting that $\mathbb{P}$ is c.c.c.
\end{proof}

For the next theorem we need to introduce a stronger concept of saturation.  Note that if $\text{NS}_{\omega_1}$ is saturated, then for any $\omega_2$-sized collection $\mathcal{S}$ of stationary subsets of $\omega_1$, there is a pair of distinct members of $\mathcal{S}$ whose intersection is stationary.  We say that $\text{NS}_{\omega_1}$ is $\boldsymbol{(\omega_2,\omega_1, < \omega)}$\textbf{-saturated} if it satisfies the following stronger requirement:  whenever $\mathcal{S}$ is an $\omega_2$-sized collection of stationary subsets of $\omega_1$, there is an $\omega_1$-sized subcollection $\mathcal{S}_0 \subset \mathcal{S}$ such that for every finite $X \subset \mathcal{S}_0$, $\bigcap X$ is stationary.  

We will make use of the following well-known lemma:
\begin{lemma}\label{lem_CBA}
If $\text{NS}_{\omega_1}$ is saturated, then $\wp(\omega_1)/\text{NS}_{\omega_1}$ is a complete boolean algebra.
\end{lemma}
\begin{proof}
Let $X$ be a collection of stationary subsets of $\omega_1$, and let $\mathcal{A}$ be a $\subseteq$-maximal antichain contained in $X$.  By saturation, $|\mathcal{A}| \le \omega_1$.  If the cardinality of $\mathcal{A}$ is exactly $\omega_1$, it is routine to show that ``the" diagonal union of $\mathcal{A}$ (using any $\omega_1$-length enumeration of $\mathcal{A}$) represents the least upper bound of $X$ in $\wp(\omega_1)/\text{NS}_{\omega_1}$.  If $|\mathcal{A}|<\omega_1$ then the union of $\mathcal{A}$ serves the same purpose.
\end{proof}

If $X$ is a collection of stationary subsets of $\omega_1$ that has a least upper bound in $\wp(\omega_1)/\text{NS}_{\omega_1}$, then we will denote this least upper bound by $\sum X$.

\begin{theorem}[Larson; cf.\ Lemma 3.11 of Dow-Tall~\cite{MR3744886}; see also Garti et al~\cite{GartiEtAl} where a slightly stronger assumption was used]\label{thm_Larson_VerySat}
Suppose $\text{NS}_{\omega_1}$ is saturated, and Projective CC holds.  Then in fact $\text{NS}_{\omega_1}$ is $(\omega_2,\omega_1, < \omega)$-saturated. 
\end{theorem}
\begin{proof}
Let $\mathcal{S}$ be an $\omega_2$-sized collection of stationary subsets of $\omega_1$; fix a one-to-one  enumeration $\vec{S}=\langle S_i \ : \ i < \omega_2 \rangle$ of $\mathcal{S}$.  For each $i < \omega_2$, let 
\[
T_i:= \sum \{ S_j \ : \ j \in [i,\omega_2) \} 
\]
Such least upper bounds exist by saturation of $\text{NS}_{\omega_1}$ and Lemma \ref{lem_CBA}.  Then $\langle T_i \ : \ i < \omega_2 \rangle$ is a descending sequence mod $\text{NS}_{\omega_1}$; so again by saturation of $\text{NS}_{\omega_1}$, it must stabilize; so there is some $i_0 < \omega_2$ such that $T_i =_{\text{NS}} T_{i_0}$ for all $i \ge i_0$.   Let $T^*:= T_{i_0}$; then 
\begin{equation}\label{eq_CompareAllToTstar}
\forall i \in [i_0,\omega_2) \ \ \ T^* =_{\text{NS}} \sum \{ S_j \ : \ j \in [i,\omega_2)  \}.
\end{equation}
By Projective CC, there is an
\[
X \prec (H_\theta,\in, \vec{S}, i_0, T^*)
\]
such that $\text{otp}(X \cap \omega_2)=\omega_1$ and $X \cap \omega_1 \in T^*$.

\begin{nonGlobalClaim}\label{clm_Omega1ManyInX}
There $\omega_1$-many $i \in X \cap \omega_2$ such that $X \cap \omega_1 \in S_i$. 
\end{nonGlobalClaim}
\begin{proof}
(of Claim \ref{clm_Omega1ManyInX})  Note that since $i_0 \in X$, $\text{otp}(X \cap \omega_2) =\omega_1$,  and $\vec{S}$ is a one-to-one enumeration, it suffices to show that for every $\gamma \in X \cap \omega_2$ such that $\gamma \ge i_0$, there is an $i \in X$ above $\gamma$ such that $X \cap \omega_1 \in S_i$.  So fix such a $\gamma$.  Then by \eqref{eq_CompareAllToTstar},
\begin{equation}\label{eq_TstarEqual}
T^* =_{\text{NS}} \ \sum \{ S_j \ : \ j \in [\gamma,\omega_2) \}.
\end{equation}
Furthermore, since $\gamma \in X$, the boolean sum on the right side of the equation is an element of $X$.  $T^*$ is also an element of $X$, by choice of $X$.  Hence the set difference
\[
T^* \ \setminus \ \sum \{ S_j \ : \ j \in [\gamma,\omega_2) \},
\]
which is nonstationary by \eqref{eq_TstarEqual}, is also an element of $X$.  It follows that $X \cap \omega_1$ cannot lie in this set difference.  But also $X \cap \omega_1 \in T^*$, by choice of $X$.  Hence
\begin{equation}\label{eq_Xcapomega_1}
X \cap \omega_1 \in \sum \{ S_j \ : \ j \in [\gamma,\omega_2) \}.
\end{equation}
Since $\vec{S}$ and $\gamma$ are elements of $X$, $\langle S_j \ : \ j \in [\gamma,\omega_2) \rangle$ is also an element of $X$.  It follows from this and \eqref{eq_Xcapomega_1} that there is some $i \in [\gamma,\omega_2) \cap X$ such that $X \cap \omega_1 \in S_i$.

\end{proof}

Let $I$ be the $\omega_1$-sized collection of indices from $X$ given by Claim \ref{clm_Omega1ManyInX}.  Consider any finite collection $i_0 < i_1, < \dots < i_n$ from $I$.  Then $S:=S_{i_0} \cap \dots \cap S_{i_n}$ is an element of $X$, and $X \cap \omega_1 \in S$.  It follows from Lemma \ref{lem_McapOmega1} that $S$ is stationary.
\end{proof}

\section{What about adjoining objects to uncountable models?}\label{sec_Higher}

This section is mostly about results of Foreman and Magidor, showing that higher versions of SCC and WRP are inconsistent.  We attempt to streamline their proof, while also highlighting the role of (the ulimately inconsistent) principle $\text{SCC}(\omega_2)$ in their arguments.

\subsection{Negative results}

The following theorem of Shelah is stated in a slightly unusual form:
\begin{theorem}[Shelah]\label{thm_Shelah_LargestCof}
Suppose $H$ is a transitive $\text{ZFC}^-$ model, $\mu \in H$ is a cardinal in $V$, $\mu^{++H}$ exists and is a cardinal in $V$, but $\mu^{+H}$ is not a cardinal in $V$.  Then $\text{cf}^V(\mu^{+H}) = \mu$. 
\end{theorem}
The proof is basically the same as Shelah's original proof; using that $H$ is a $\text{ZFC}^-$ model that believes $\mu^{++}$ exists, $H$ has a strongly almost disjoint, $\mu^{++H}$-sized family of subsets of $\mu^{+H}$, and this is upward absolute to $V$.  Shelah's argument then shows that $\mu^{+H}$ cannot have cofinality strictly less than $\mu$ (see Lemma 23.19 of \cite{MR1940513}).

\begin{theorem}[Foreman-Magidor~\cite{MR1359154}]\label{thm_FM_Chang}
There is an $F:[\omega_3]^{<\omega} \to \omega_3$ such that whenever $X \subset \omega_3$ is closed under $F$, $|X| = \omega_2$, and $X \cap \omega_2$ is an ordinal in the interval $(\omega_1,\omega_2)$, then $X \cap \omega_2$ is $\omega_1$-cofinal.
\end{theorem}
\begin{proof}
If there were no such $F$, then there would be (weakly) stationarily many $X \subset \omega_3$ such that $|X|=\omega_2$ and $X \cap \omega_2$ is an $\omega$-cofinal ordinal in $(\omega_1,\omega_2)$.  Let $S$ denote this stationary set.  By Lemma \ref{lem_DontEscapeMcapW}, there exists a
\[
Y \prec (H_{\omega_4},\in)
\]
such that $Y \cap \omega_3 \in S$.  Fix such a $Y$.  Since $Y \cap \omega_3 \in S$, then by definition of $S$, it follows that
\begin{equation}\label{eq_YcapOmega_2}
Y \cap \omega_2 \text{ is an } \omega \text{-cofinal ordinal in } (\omega_1,\omega_2).
\end{equation}
   Now $|Y \cap \omega_3|=\omega_2$, but in fact $Y \cap \omega_3$ must have ordertype exactly (i.e.\ no larger than) $\omega_2$.\footnote{To prove this, consider an arbitrary $\eta \in Y \cap \omega_3$.  Since $Y \prec (H_{\omega_4},\in)$, there is a surjection $f: \omega_2 \to \eta$ with $f \in Y$, and hence $Y \cap \eta = f[Y \cap \omega_2]$; the latter set has cardinality $\omega_1$.  In short, every proper initial segment of $Y \cap \omega_3$ has cardinality, and hence ordertype, $<\omega_2$.}

Let $\sigma: H_Y \to H_{\omega_4}$ be the inverse of the transitive collapse of $Y$.  The calculations above regarding $Y$'s trace on $\omega_3$ imply that 
\[
\omega_1^V = \omega_1^{H_Y} < \omega_2^{H_Y} = \text{crit}(\sigma) < \omega_3^{H_Y} = \omega_2^V.
\]
 Theorem \ref{thm_Shelah_LargestCof} implies that $\omega_2^{H_Y}=Y \cap \omega_2$ is $\omega_1$-cofinal, contradicting \eqref{eq_YcapOmega_2}.
\end{proof}

\begin{corollary}\label{cor_SCC_omega2_incons}
$\text{SCC}(\omega_2)$ is inconsistent.  (Recall this notion was defined on page \pageref{def_SCC_localversions}).
\end{corollary}
\begin{proof}
Assume toward a contradiction that $\text{SCC}(\omega_2)$ holds.  Fix a large $\theta$ and a wellorder $\Delta$ of $H_\theta$, and let $F$ be the $\Delta$-least function satisfying the conclusion of Theorem \ref{thm_FM_Chang}.  Fix an $M \prec (H_\theta,\in,\Delta)$ such that $|M|=\omega_1 \subset M$ and $M \cap \omega_2$ is an $\omega$-cofinal ordinal in the interval $(\omega_1,\omega_2)$.  Using $\text{SCC}(\omega_2)$, build a $\subseteq$-increasing and continuous chain $\langle M_i \ : \ i < \omega_2 \rangle$ such that $M = M_0$, $M_i \subset M_{i+1}$, $M_{i+1} \cap \omega_2 = M_{i} \cap \omega_2$, $(M_{i+1} \setminus M_i)  \cap \omega_3 \ne \emptyset$, $|M_i| = \omega_1$, and $M_i \prec (H_\theta,\in,\Delta)$ for all $i < \omega_2$.  Let $Y$ be the union of the $M_i$'s.  Then $Y \prec (H_\theta,\in,\Delta)$, $|Y \cap \omega_3|=\omega_2$, and $Y \cap \omega_2 = M \cap \omega_2$ is an $\omega$-cofinal ordinal.  But $F \in Y$ and hence $Y$ is closed under $F$.  This contradicts Theorem \ref{thm_FM_Chang}. 
\end{proof}

Corollary \ref{cor_SCC_omega2_incons} and Theorem \ref{thm_WRP_implies_SCC} imply:

\begin{corollary}[Foreman-Magidor~\cite{MR1359154}]
The principle $\text{WRP}\big( \wp^*_{\omega_2} \big)$ is inconsistent.
\end{corollary}

%
%
%

\subsection{Positive results}



While $\text{WRP}\big( \wp^*_{\omega_2} \big)$ is always false, a restricted version of it is consistent.  Recall the class $\text{IA}_{\omega_1}$ from Fact \ref{fact_IA}.

\begin{theorem}[Foreman-Magidor~\cite{MR1359154}]\label{thm_FM_PossibleHigherWRPIA}
If $\kappa$ is supercompact, then the following statement holds after forcing with the Levy collapse $\text{Col}(\omega_2, < \kappa)$:  For every regular $\theta \ge \omega_3$ and every stationary $S \subseteq \wp^*_{\omega_2}(H_\theta)$ such that
\[
S \subseteq \text{IA}_{\omega_1},
\]
there is a $W \in \wp^*_{\omega_3}(H_\theta)$ such that $S \cap \wp^*_{\omega_2}(W)$ is stationary in $\wp^*_{\omega_2}(W)$.
\end{theorem}

\section{Open Problem Summary}

Here we collect the open problems that were mentioned at various places in the survey:
\begin{enumerate}
 \item (From Section \ref{sec_LocalSCC}): Are any of the following implications reversible?
\[
\xymatrix{
\text{SCC}^{\text{cof}} \ar@{=>}[r] \ar@/_2pc/@{=>}[rr] & \text{SCC}^{\text{split}} \ar@{=>}[r] & \text{SCC}
}
\]

\vspace{2pc} 

\noindent The author conjectures that none of the implications are reversible.  The question of whether the implication $\text{SCC}^{\text{cof}} \implies \text{SCC}$ is reversible was also asked by Usuba~\cite{MR3248209}.
 \item (From Section \ref{sec_RelationSCC_WRP}):  Does WRP imply $\text{RP}_{\text{internal}}$?  This is similar to questions of Krueger~\cite{MR2674000} (whether WRP implies the principle $\text{RP}_{\text{IS}}$) and Beaudoin~\cite{MR877870} (whether WRP implies Fleissner's Axiom R).  See Cox~\cite{Cox_RP_IS} for more details.
 \item (From Section \ref{sec_SCC_TP}):  Does $\neg \text{CH}$ plus ``every stationary subset of $[\omega_2]^\omega$ reflects to a set of size $\omega_1$" imply that the Tree Property holds at $\omega_2$?  This question also appears in \cite{MR3431031}.
\end{enumerate}


\bib{MR2768684}{article}{
  author={Abraham, Uri},
  title={Proper forcing},
  conference={ title={Handbook of set theory. Vols. 1, 2, 3}, },
  book={ publisher={Springer, Dordrecht}, },
  date={2010},
  pages={333--394},
}

\bib{MR654852}{article}{
  author={Baumgartner, James E.},
  author={Taylor, Alan D.},
  title={Saturation properties of ideals in generic extensions. II},
  journal={Trans. Amer. Math. Soc.},
  volume={271},
  date={1982},
  number={2},
  pages={587--609},
  issn={0002-9947},
  review={\MR {654852 (83k:03040b)}},
  doi={10.2307/1998900},
}

\bib{MR877870}{article}{
  author={Beaudoin, Robert E.},
  title={Strong analogues of Martin's axiom imply axiom R},
  journal={J. Symbolic Logic},
  volume={52},
  date={1987},
  number={1},
  pages={216--218},
  issn={0022-4812},
  review={\MR {877870}},
  doi={10.2307/2273877},
}

\bib{Cox_Nonreasonable}{article}{
  author={Cox, Sean},
  title={Chang's conjecture and semiproperness of nonreasonable posets},
  journal={Monatsh. Math.},
  volume={187},
  date={2018},
  number={4},
  pages={617--633},
}

\bib{Cox_RP_IS}{article}{
  author={Cox, Sean},
  title={Forcing axioms, approachability at $\omega _2$, and stationary set reflection},
  status={to appear},
  journal={Journal of Symbolic Logic},
  eprint={https://arxiv.org/abs/1807.06129},
}

\bib{Cox_Sakai_SCC}{article}{
  author={Cox, Sean},
  author={Sakai, Hiroshi},
  title={A variant of Shelah's characterization of Strong Chang's conjecture},
  journal={MLQ Math. Log. Q.},
  volume={65},
  date={2019},
  number={2},
  pages={251--257},
}

\bib{MR3065118}{article}{
  author={Doebler, Philipp},
  title={Rado's conjecture implies that all stationary set preserving forcings are semiproper},
  journal={J. Math. Log.},
  volume={13},
  date={2013},
  number={1},
  pages={1350001, 8},
  issn={0219-0613},
  review={\MR {3065118}},
  doi={10.1142/S0219061313500013},
}

\bib{MR2576698}{article}{
  author={Doebler, Philipp},
  author={Schindler, Ralf},
  title={$\Pi _2$ consequences of BMM plus $\text {NS}_{\omega _1}$ is precipitous and the semiproperness of stationary set preserving forcings},
  journal={Math. Res. Lett.},
  volume={16},
  date={2009},
  number={5},
  pages={797--815},
  issn={1073-2780},
}

\bib{MR3744886}{article}{
  author={Dow, Alan},
  author={Tall, Franklin D.},
  title={Normality versus paracompactness in locally compact spaces},
  journal={Canad. J. Math.},
  volume={70},
  date={2018},
  number={1},
  pages={74--96},
  issn={0008-414X},
  review={\MR {3744886}},
  doi={10.4153/CJM-2017-006-x},
}

\bib{MR946635}{article}{
  author={Feng, Qi},
  title={On weakly stationary sets},
  journal={Proc. Amer. Math. Soc.},
  volume={105},
  date={1989},
  number={3},
  pages={727--735},
  issn={0002-9939},
  review={\MR {946635}},
  doi={10.2307/2046926},
}

\bib{MR1668171}{article}{
  author={Feng, Qi},
  author={Jech, Thomas},
  title={Projective stationary sets and a strong reflection principle},
  journal={J. London Math. Soc. (2)},
  volume={58},
  date={1998},
  number={2},
  pages={271--283},
  issn={0024-6107},
  review={\MR {1668171 (2000b:03166)}},
}

\bib{MattHandbook}{article}{
  author={Foreman, Matthew},
  title={Ideals and generic elementary embeddings},
  conference={ title={Handbook of set theory. Vols. 1, 2, 3}, },
  book={ publisher={Springer, Dordrecht}, },
  date={2010},
  pages={885--1147},
}

\bib{MR1359154}{article}{
  author={Foreman, Matthew},
  author={Magidor, Menachem},
  title={Large cardinals and definable counterexamples to the continuum hypothesis},
  journal={Ann. Pure Appl. Logic},
  volume={76},
  date={1995},
  number={1},
  pages={47--97},
  issn={0168-0072},
  review={\MR {1359154 (96k:03124)}},
}

\bib{MR924672}{article}{
  author={Foreman, M.},
  author={Magidor, M.},
  author={Shelah, S.},
  title={Martin's maximum, saturated ideals, and nonregular ultrafilters. I},
  journal={Ann. of Math. (2)},
  volume={127},
  date={1988},
  number={1},
  pages={1--47},
}

\bib{MR2115072}{article}{
  author={Foreman, Matthew},
  author={Todorcevic, Stevo},
  title={A new L\"owenheim-Skolem theorem},
  journal={Trans. Amer. Math. Soc.},
  volume={357},
  date={2005},
  number={5},
  pages={1693--1715 (electronic)},
  issn={0002-9947},
  review={\MR {2115072 (2005m:03064)}},
  doi={10.1090/S0002-9947-04-03445-2},
}

\bib{FuchinoUsuba}{article}{
  author={Fuchino, Saka{\'e}},
  author={Usuba, Toshimichi},
  title={A reflection principle formulated in terms of games},
  journal={RIMS Kokyuroku},
  number={1895},
  pages={37--47},
}

\bib{GartiEtAl}{article}{
  title={Martin's maximum and the non-stationary ideal},
  author={Garti, Shimon},
  author={Hayut, Yair},
  author={Horowitz, Haim},
  author={Magidor, Menachem},
  journal={arXiv preprint arXiv:1708.08049},
  year={2017},
}

\bib{MR820120}{article}{
  author={Gitik, Moti},
  title={Nonsplitting subset of ${\scr P}\sb \kappa (\kappa \sp +)$},
  journal={J. Symbolic Logic},
  volume={50},
  date={1985},
  number={4},
  pages={881--894 (1986)},
}

\bib{MR1940513}{book}{
  author={Jech, Thomas},
  title={Set theory},
  series={Springer Monographs in Mathematics},
  note={The third millennium edition, revised and expanded},
  publisher={Springer-Verlag},
  place={Berlin},
  date={2003},
  pages={xiv+769},
  isbn={3-540-44085-2},
  review={\MR {1940513 (2004g:03071)}},
}

\bib{MR560220}{article}{
  author={Jech, T.},
  author={Magidor, M.},
  author={Mitchell, W.},
  author={Prikry, K.},
  title={Precipitous ideals},
  journal={J. Symbolic Logic},
  volume={45},
  date={1980},
  number={1},
  pages={1--8},
  issn={0022-4812},
  review={\MR {560220 (81h:03097)}},
}

\bib{MR520190}{article}{
  author={Kanamori, A.},
  author={Magidor, M.},
  title={The evolution of large cardinal axioms in set theory},
  conference={ title={Higher set theory (Proc. Conf., Math. Forschungsinst., Oberwolfach, 1977)}, },
  book={ series={Lecture Notes in Math.}, volume={669}, publisher={Springer, Berlin}, },
  date={1978},
  pages={99--275},
  review={\MR {520190}},
}

\bib{MR2674000}{article}{
  author={Krueger, John},
  title={Internal approachability and reflection},
  journal={J. Math. Log.},
  volume={8},
  date={2008},
  number={1},
  pages={23--39},
}

\bib{MR2069032}{book}{
  author={Larson, Paul B.},
  title={The stationary tower},
  series={University Lecture Series},
  volume={32},
  note={Notes on a course by W. Hugh Woodin},
  publisher={American Mathematical Society},
  place={Providence, RI},
  date={2004},
  pages={x+132},
  isbn={0-8218-3604-8},
  review={\MR {2069032 (2005e:03001)}},
}

\bib{MR2817562}{article}{
  author={Sharpe, I.},
  author={Welch, P. D.},
  title={Greatly Erd\H os cardinals with some generalizations to the Chang and Ramsey properties},
  journal={Ann. Pure Appl. Logic},
  volume={162},
  date={2011},
  number={11},
  pages={863--902},
  issn={0168-0072},
  review={\MR {2817562}},
  doi={10.1016/j.apal.2011.04.002},
}

\bib{MR1623206}{book}{
  author={Shelah, Saharon},
  title={Proper and improper forcing},
  series={Perspectives in Mathematical Logic},
  edition={2},
  publisher={Springer-Verlag},
  place={Berlin},
  date={1998},
  pages={xlviii+1020},
  isbn={3-540-51700-6},
  review={\MR {1623206 (98m:03002)}},
}

\bib{MR1261218}{article}{
  author={Todor{\v {c}}evi{\'c}, Stevo},
  title={Conjectures of Rado and Chang and cardinal arithmetic},
  conference={ title={Finite and infinite combinatorics in sets and logic (Banff, AB, 1991)}, },
  book={ series={NATO Adv. Sci. Inst. Ser. C Math. Phys. Sci.}, volume={411}, publisher={Kluwer Acad. Publ., Dordrecht}, },
  date={1993},
  pages={385--398},
  review={\MR {1261218}},
}

\bib{MR3431031}{article}{
  author={Torres-P{\'e}rez, V{\'{\i }}ctor},
  author={Wu, Liuzhen},
  title={Strong Chang's conjecture and the tree property at $\omega _2$},
  journal={Topology Appl.},
  volume={196},
  date={2015},
  number={part B},
  part={part B},
  pages={999--1004},
  issn={0166-8641},
  review={\MR {3431031}},
}

\bib{MR3600760}{article}{
  author={Torres-P\'{e}rez, V\'{\i }ctor},
  author={Wu, Liuzhen},
  title={Strong Chang's conjecture, semi-stationary reflection, the strong tree property and two-cardinal square principles},
  journal={Fund. Math.},
  volume={236},
  date={2017},
  number={3},
  pages={247--262},
  issn={0016-2736},
  review={\MR {3600760}},
  doi={10.4064/fm257-5-2016},
}

\bib{MR3248209}{article}{
  author={Usuba, Toshimichi},
  title={Bounded dagger principles},
  journal={MLQ Math. Log. Q.},
  volume={60},
  date={2014},
  number={4-5},
  pages={266--272},
  issn={0942-5616},
  review={\MR {3248209}},
  doi={10.1002/malq.201300019},
}

\bib{Weiss_CombEssence}{article}{
  author={Wei{\ss }, Christoph},
  title={The combinatorial essence of supercompactness},
  journal={Ann. Pure Appl. Logic},
  volume={163},
  date={2012},
  number={11},
  pages={1710--1717},
}



\begin{bibdiv}
\begin{biblist}
\bibselect{../../../MasterBibliography/Bibliography}
\end{biblist}
\end{bibdiv}
\end{document}